\documentclass[11pt]{article}

\usepackage{graphicx}
\usepackage{color}
\usepackage{amssymb,latexsym,amsmath, upgreek}
\usepackage{epstopdf}
\usepackage{dsfont}

\usepackage[T1]{fontenc}
\usepackage{hyperref}
\usepackage{url}

\newtheorem{lemma}{Lemma}
\newtheorem{theorem}{Theorem}
\newtheorem{proposition}{Proposition}
\newtheorem{claim}{Claim}

\newtheorem{problem}{Problem}
\newtheorem{corollary}{Corollary}
\newtheorem{remark}{Remark}

\newcommand{\G}{\Gamma}
\newcommand{\Om}{\Omega}
\newcommand{\GG}{\mathbb{G}}
\newcommand{\PP}{\mathbb{P}}

\newcommand{\forme}[1]{}

\def\wbull{\hfill\vrule height .9ex width .8ex depth -.1ex}

\begin{document}

 \begin{center}
 {\Large \bf On the Sombor index of graphs with given connectivity and number of bridges}\\

 \vspace{6mm}

 \bf {Sakander Hayat$^a$, Muhammad Arshad$^b$, Kinkar Chandra Das$^{c,*}$}

 \vspace{5mm}

 {\it $^a$Faculty of Science, Universiti Brunei Darussalam,\\
  Jln Tungku Link, Gadong BE1410, Brunei Darussalam\/} \\
{\tt E-mail: sakander1566@gmail.com}\\[2mm]

{\it $^b$Faculty of Engineering Sciences, GIK Institute of Engineering Sciences and Technology,\\
Topi, Swabi, 23460, Pakistan\/} \\
{\tt E-mail: arshadswabi@outlook.com}\\[2mm]

{\it $^c$Department of Mathematics, Sungkyunkwan University, \\
 Suwon 16419, Republic of Korea\/} \\
{\tt E-mail: kinkardas2003@gmail.com}

 \hspace*{30mm}

 \end{center}

 \baselineskip=0.23in
	
\begin{abstract}

Recently in 2021, Gutman introduced the Sombor index of a graph, a novel degree-based topological index.
It has been shown that the Sombor index efficiently models the thermodynamic properties of chemical compounds.
Assume $\mathbb{B}_n^k$ (resp. $\mathbb{V}_n^k$) comprises all graphs with order $n$ having
number of bridges (resp. vertex-connectivity) $k$. Horoldagva \& Xu (2021) characterized
graphs achieving the maximum Sombor index of graphs in $\mathbb{B}_n^k$.
This paper characterizes graphs achieving the minimum Sombor index in $\mathbb{B}_n^k$.
Certain auxiliary operation on graphs in $\mathbb{B}_n^k$ are introduced and employed for the characterization.
Moreover, we characterize graphs achieving maximum Sombor index in $\mathbb{V}_n^k$.
Some open problems, which naturally arise from this work, have been proposed at the end.\\

\noindent
{\bf 2020 MCS:} 05C35, 05C09, 05C92.\\[1mm]
\end{abstract}
	
\baselineskip=0.28in

\section{Introduction}\label{sec1}

Dependence of physicochemical properties of a compound on its molecular structure
is a cornerstone idea in contemporary chemistry. Information regarding this structural
dependence is frequently retrieved by employing various advance mathematical methods
such as graph-theoretic topological descriptors. These numerical invariants have diverse
amount of potential applications in chemistry, pharmacy, material science, among others,
\cite{GF2010,TC2000}. Degree-based topological invariants \cite{GutmanDegree2013} are structured based on degrees
of vertices in a chemical graphs. Their mathematical and computational simplicity make
them one of most successful class of graphical indices \cite{GF2010}.
Lately, a novel degree-based index known as the Sombor index is put forwarded by
Gutman in 2021. The defining structure of the Sombor index is:
\begin{equation*}
SO(\G)=\sum_{uv\in E(\G)}\sqrt{\mathrm{deg}_\G(u)^2+\mathrm{deg}_\G(v)^2},
\end{equation*}
in which $\mathrm{deg}_\G(x)$ refers to the valency/degree of $x\in V(\G)$.\\

Immediately after its inception, the chemical applicability of the Sombor index has been studied.
Red\v{z}epovi\'{c} \cite{Red2021} introduced the Sombor index in quantitative structure property
relationship (QSPR) modeling of entropy and heat of vaporization of alkanes. Experimental results
of the study asserted that the Sombor index possesses a strong correlation potential with these physicochemical
properties. Moreover, the Sombor index has been proven to have potential in simulating thermodynamic
properties of hydrocarbons and it showed a significant potential in the simulation.
After its chemical applicability was reported, the Sombor index has been studied by
mathematicians as well as theoretical chemists. Various mathematical properties of the Sombor index
have been reported, for instance, in \cite{AG2021,CLW2022,CGR2021,CR2021,DCCS2021,DGA1,DG2022,DS1,FYL2021,Filipovski2021,Gutman1,GutmanTrees,
Gutman2,HX2021,LWZ2022,Lin2021,Liu2021a,Liu2021b,LYH2022,LYTL2021,MMAM2021,MMM2021,NSW2022,
RDA2021,Shang2022,WMLF2021,YJDM}.\\

Although an extensive amount of literature is available on various mathematical properties of Sombor index,
characterization of graphs attaining maximal/minimal Sombor index of graphs with a fix property has captured researchers'
attention as of late. For example, Hayat \& Rehman \cite{HR2022} studied minimum Sombor index of graphs with 
fixed amount of articulation/cut-vertices.
Sun \& Du \cite{SD2022}  characterized extremal Sombor index of graphs having give the domination number in prior.
Trees and unicylic graphs with fixed maximum degree and matching number achieving the maximum/minimum
Sombor index were classified by Zhou et al. \cite{ZLM2021a,ZLM2021b}. Recently in 2022, Aashtab et al. \cite{AAMNS2022}
proved a result asserting that the difference among the minimum and minimum valencies attaining the minimum Sombor
index is not greater than 1. A rich survey on bounds and other extremal results has been put forwarded by
Liu et al. \cite{LGYH2022}.\\

Studying extremal graphs with respect to various topological descriptors of graphs 
with given vertex-connectivity as well as possessing a fixed amount of bridges and has been an active area of contemporary research.
Feng et al. \cite{FHL2010} studied the first maximum/minimum graphs corresponding to Zagreb indices in
$\mathbb{B}_n^k$ i.e. $n$-vertex graphs with $k$ bridges.
Chen \& Liu \cite{CL2014} studied the second and third minimum and maximum
graphs respective to Zagreb indices in $\mathbb{B}_n^k$. Horoldagva et al. \cite{HBD2017}
studied extremal reduced version of the second Zagreb and the difference among $1^{\mathrm{st}}$
\& $2^{\mathrm{nd}}$ Zagreb indices of graphs $\mathbb{B}_n^k$. Amin \& Nayeem \cite{AN2020} recently
found extremal forgotten topological index of graphs in $\mathbb{B}_n^k$.
Wang \cite{Wang2010} investigated cacti possessing maximum Merrifield-Simmons index having fixed amount of bridges.
Wang et al. \cite{WWC2017} studied extremal multiplicative Zagreb indices of graphs in $\mathbb{B}_n^k$.
Regarding extremal topological indices of graphs in $\mathbb{V}_n^k$, Tomescu et al. \cite{Tomescu2015}
studied extremal values for general Randi\'{c}/sum-connectivity indices.
Li \& Fan \cite{LiFan2015} derived extremal Harary index of graphs in $\mathbb{V}_n^k$.
Zhang et al. \cite{Zhang2016} studied maximum atom-bond connectivity index of graphs in $\mathbb{V}_n^k$.\\

Recently in 2021, Horoldagva \& Xu \cite{HX2021} classified
graphs achieving the maximum Sombor index of graphs within $\mathbb{B}_n^k$.
This paper characterizes graphs attaining minimum Sombor index
within $\mathbb{B}_n^k$. Extremal graphs achieving the bound
have been classified.
Moreover, a sharp upper bound on the Sombor index of graphs within $\mathbb{V}_n^k$ has been found.
Extremal graphs attaining the bound have been classified.

\section{Preliminaries}
	
Denoted by $\G=(V,E)$, a graph comprises sets $V$ of vertices and $E$ of edges.
The number $n$ (resp. $\epsilon$) is frequently used to denote the number of vertices (resp. edges)
i.e. order (resp. size) of $\G$.  We use $(n,\epsilon)-$graph to symbolize an $n$-vertex graph
comprising $\epsilon$ number of edges. For any $y\in V(\G)$, let $\mathrm{deg}_\G(y)$ be the
number of vertices connected to $y$ by an edge. A vertex- (resp. edge-) deleted subgraph $\G-xy$
(resp. $\G-y$) is constructed by removing $xy\in E(\G)$ (resp. $y\in V(\G)$ and incident edges) from $\G$.
We denote the $n$-vertex cycle (resp. path) by $C_n$ (resp. $P_n$).\\

The number of edges traversed by a path is said to be its length. For a pair $x,y \in V(\G)$,
the distance $d(x,y)$ is a shortest path's length between them.
The diameter $D(\G)$ of $\G$ i.e. $D(\G)=\max\{d(x,y):x,y\in V(\G)\}$ is the largest distance in $\G$.
An $x-y$ geodesic is a path having length $d(x,y)$.
A pair of vertices $x,y\in V(\G)$ are said to be antipodal if $d(x,y) = D(\G)$.
A vertex $x$ in $\G$ in said to be cut-vertex of $\G$ if $\G-x$ is disconnected.
Similarly, an edge $xy$ in $\G$ is called a bridge (or cut-edge), if $\G-xy$ has two components.
A edge cut (resp. vertex cut) of $\G$ is $T\subseteq E(\G)$ (resp. $S\subseteq V(\G)$)
whose removal results in $\G$ to be disconnected. The minimum size $k$ of a vertex-cut (resp. edge-cut)
is known as the vertex-connectivity (resp. edge-connectivity).\\

Next, we present some preliminary results which are employed in subsequent sections.
The following is a standard result in graph theory.
\begin{theorem}\emph{\cite{West2001}}\label{bridgecyclesthm}
An edge $e\in E(\G)$ in $\G$ is a bridge if and only if no cycle of $\G$ contains $e$.
\end{theorem}
Theorem \ref{bridgecyclesthm} suggest the following corollary.
\begin{corollary}\label{Tree_bridge}
A graph $\G$ of order $n$ is a tree if and only if $\G$ has $n-1$ bridges.
\end{corollary}
Since a $(n,\epsilon)$-graph is a tree if and only if $\epsilon=n-1$. Thus, we obtain:
\begin{proposition}\emph{\cite{West2001}}\label{connectedgraphwithcycles}
Assume $\G$ is an $(n,\epsilon)$-graph having at least one cycle. Then $\epsilon\geq n$.
\end{proposition}
The following result shows that antipodal vertices in a tree are the pendent vertices.
\begin{proposition}\emph{\cite{West2001}}\label{endvertex}
Let $x,y\in V(\G)$ be two antipodal vertices in a tree $T$. Then, $\mathrm{deg}_T(x)=\mathrm{deg}_T(y)=1$.
\end{proposition}
The definition of the Sombor index of a graph suggests the following elementary observation.
\begin{lemma}\label{lem0}
For an edge-deleted subgraph $\G-xy$, where $xy\in E(\G)$, we have $SO(\G)>SO(\G-xy)$.
\end{lemma}

Next, we prove some auxiliary results that are subsequently employed in the main results.

\section{Auxiliary Results}\label{auxualiary}

For the Sombor index of graphs within $\mathbb{B}_n^k$, to find a sharp upper bound, we would need
some auxiliary operation which we introduce herewith.

First, on a tree $T$, we introduce the $\tau$-operation.
Assume $T$ to be a tree comprising vertices $x_1, y_1,u_1, \dots, u_l\in V(T)$ where
$u_r u_{r-1}, u_r u_{r+1}, u_r x_1, u_r y_1 \in E(T)$ and $\mathrm{deg}_{T}(u_r) \geq 3$.
Let $T_0$ and $T_1$ be the subtrees identified to vertices $x_1$ and $y_1$, respectively.
Let $T_{\tau}$ be the tree obtained from $T$ by adding (resp. removing) the edge $u_1x_1$
(resp. $u_r x_1$). We call the tree $T_{\tau}$, the $\tau$-switched tree. See Figure \ref{Fig1}
for an application of the $\tau$-operation on the tree $T$.

\begin{figure}[htbp!]
\centering
\includegraphics[scale=1.2]{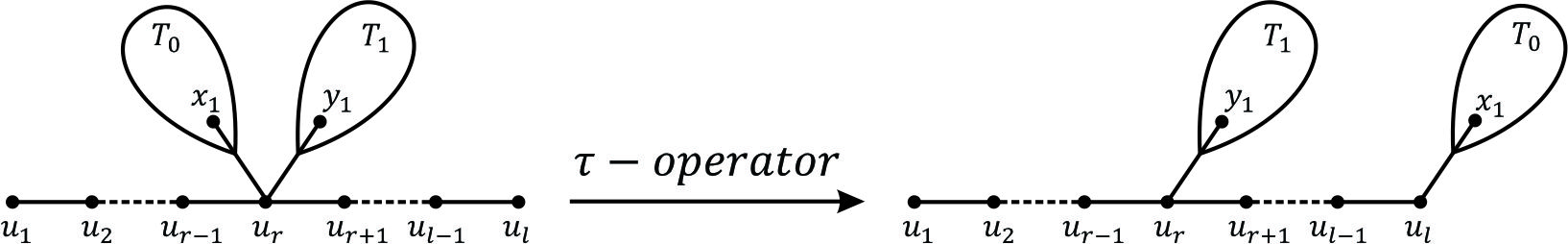}
\caption{Applying the $\tau$-operation on the tree $T$.}\label{Fig1}
\end{figure}

We conduct a comparison between Sombor indices of a tree $T$ and its $\tau$-switched tree i.e. $T_{\tau}$.
\begin{lemma}\label{lem00}
Assume $T_{\tau}$ to be a tree obtained by applying the $\tau$-operation on a tree $T$ as explained by
Figure \ref{Fig1}.
We have $$SO(T_{\tau})<SO(T).$$
\end{lemma}
\begin{proof}
Note that the tree $T$ consists of vertices $x_1, y_1,u_1, \dots, u_l\in V(T)$ where
$u_r x_1, u_r y_1, u_r u_{r-1}, u_r u_{r+1} \in E(T)$ and $\mathrm{deg}_{T}(u_r) \geq 3$.
The vertices $u_1$ and $u_l$ are two antipodal vertices of $T$. Proposition \ref{endvertex}
delivers that $\mathrm{deg}_{T}(u_l)=\mathrm{deg}_{T}(u_1)=1$. Let the neighborhood of $u_r\in V(T)$
be $N_T(u_r)=\{u_{r-1},u_{r+1},x_1,y_1\}$. The vertex $y_1$ (resp. $x_1$) is adjoined to the subtree
$T_1$ (resp. $T_0$). Figure \ref{Fig1} presents the tree $T$ and the $\tau$-operation applying on it.\\
		
The application of $\tau$-operation on $T$ suggests adding and edge $u_l x_1$ and removing $u_r x_1\in E(T)$.
This implies the shifting of $T_0$ from the vertex $x_1$ to $u_l$. 
Between $SO(T)$ and $SO(T_{\tau})$, this shifting suggests an emergence of the following
relation:
\begin{eqnarray*}
SO(T)-SO(T_{\tau})&=&\sqrt{\mathrm{deg}_T(u_r)^2 + \mathrm{deg}_T(u_{r-1})^2} + \sqrt{\mathrm{deg}_T(u_r)^2 + \mathrm{deg}_T(u_{r+1})^2}+\\
& & \sqrt{\mathrm{deg}_{T}(u_r)^2 + \mathrm{deg}_{T}(y_1)^2} + \sqrt{\mathrm{deg}_{T}(u_r)^2 + \mathrm{deg}_{T}(x_1)^2} + \\
& & \sqrt{\mathrm{deg}_{T}(u_{l-1})^2 + \mathrm{deg}_{T}(u_l)^2}-\sqrt{\mathrm{deg}_{T_{\tau}}(u_r)^2 + \mathrm{deg}_{T_{\tau}}(u_{r-1})^2}-\\
& & \sqrt{\mathrm{deg}_{T_{\tau}}(u_r)^2 + \mathrm{deg}_{T_{\tau}}(u_{r+1})^2} - \sqrt{\mathrm{deg}_{T_{\tau}}(u_r)^2 + \mathrm{deg}_{T_{\tau}}(y_1)^2}- \\
& & \sqrt{\mathrm{deg}_{T_{\tau}}(u_{l-1})^2 + \mathrm{deg}_{T_{\tau}}(u_l)^2} - \sqrt{\mathrm{deg}_{T_{\tau}}(u_l)^2 + \mathrm{deg}_{T_{\tau}}(x_1)^2},\\
&\geq& \sqrt{4^2 + 2^2} + \sqrt{4^2 + 2^2} + \sqrt{4^2 + 2^2} + \sqrt{4^2 + 2^2} + \sqrt{2^2 + 1^2} - \\
& & \sqrt{3^2 + 2^2} - \sqrt{3^2 + 2^2} - \sqrt{3^2 + 2^2} - \sqrt{2^2 + 2^2} - \sqrt{2^2 + 2^2},\\
&=& 4\sqrt{20} + \sqrt{5} -3\sqrt{13} - 2\sqrt{8}>0.
\end{eqnarray*}
This completes the proof.
\end{proof}
\wbull \\

Now we introduce the $\alpha$-operation by employing it to a $\G \in \mathbb{B}_n^k$. Assume that $\G$ 
comprises vertices $u_1,\ldots,u_r$, $v_1,\ldots,v_l$ and $w_1,\ldots,w_{k-1}$ such that $u_1,\ldots,u_r$ 
(resp. $v_1,\ldots,v_l$) forms a cycle $C_r$ (resp. $C_l$), $w_i$ ($1\leq i \leq k-1$) induces
a path, say, $P_{k-1}$, and $u_1w_1,v_1 w_{k-1}\in E(\G)$. Let 
$\G_{\alpha}=\G-\{u_1 u_2, v_1 v_2,v_1 v_l\}+\{u_2 v_2,u_1 v_l\}$.
The resulting graph $\G_{\alpha}$ obtained by applying the $\alpha$-operation
is called the $\alpha$-switched graph.
Figure \ref{Fig2} employs the $\alpha$-operation on $\G \in \mathbb{B}_n^k$.
	
\begin{figure}[htbp!]
\centering
\includegraphics[scale=0.9]{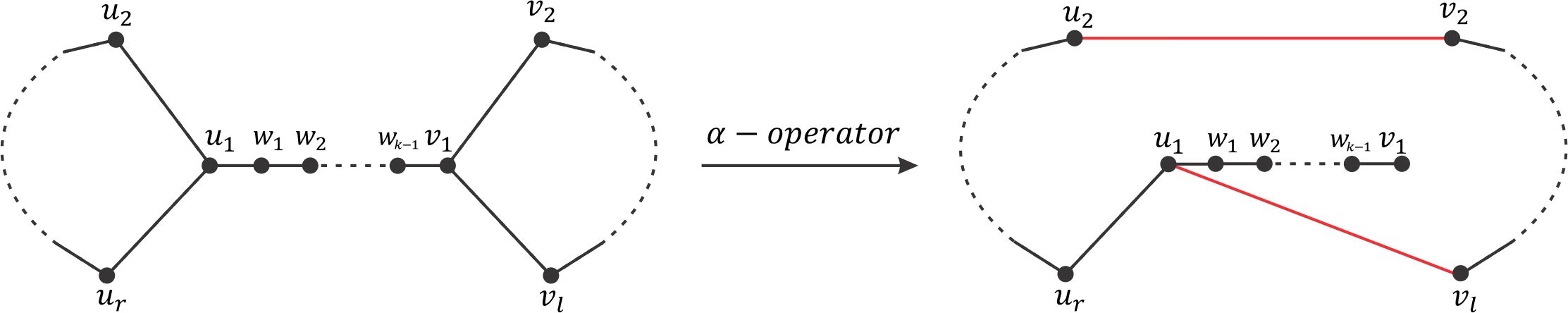}
\caption{Application of the $\alpha$-operation on $\G \in \mathbb{B}_n^k$.
It has been used in Lemma \ref{lemmaalpha}.}\label{Fig2}
\end{figure}

\begin{lemma}\label{lemmaalpha}
Assume $\G_{\alpha}$ is obtained by applying the $\alpha$-operation on $\G\in\mathbb{B}_n^k$ as explained in Figure \ref{Fig2}.
We have $$SO(\G_{\alpha})<SO(\G).$$
\end{lemma}
\begin{proof}
Notice that the graph $\G$ consists of two cycles, $C_r = u_1 u_2 \dots u_r$ and $C_l = v_1 v_2 \dots u_l$ and
a path $P_{k-1} = (w_1, w_2, \dots , w_{k-1})$ such that $u_1 \sim w_1$ and $w_{k-1} \sim v_1$. Let
$\G_\alpha$ be a graph obtained from $\G$ by an addition of edges $u_2 v_2$  and  $u_1 v_l$ and
a deletion of edges $u_1 u_2, v_1 v_2$ and $v_1 v_l$. 
Between $SO(\G)$ and $SO(\G_{\alpha})$, this shifting suggests an emergence of the following
relation:
\begin{eqnarray*}
SO(\G)-SO(\G_\alpha) &=& \sqrt{{\mathrm{deg}_{\G}(u_1)}^2 + {\mathrm{deg}_{\G}(u_2)}^2} + \sqrt{{\mathrm{deg}_{\G}(u_1)}^2 + {\mathrm{deg}_{\G}(u_r)}^2} + \sqrt{{\mathrm{deg}_{\G}(u_1)}^2 + {\mathrm{deg}_{\G}(w_1)}^2} +\\
&& \sqrt{{\mathrm{deg}_{\G}(w_{k-1})}^2 + {\mathrm{deg}_{\G}(v_1)}^2} + \sqrt{{\mathrm{deg}_{\G}(v_1)}^2 + {\mathrm{deg}_{\G}(v_2)}^2} + \sqrt{{\mathrm{deg}_{\G}(v_1)}^2 + {\mathrm{deg}_{\G}(v_l)}^2} -\\
&& \sqrt{{\mathrm{deg}_{\G_\alpha}(u_2)}^2 + {\mathrm{deg}_{\G_\alpha}(v_2)}^2} - \sqrt{{\mathrm{deg}_{\G_\alpha}(u_1)}^2 + {\mathrm{deg}_{\G_\alpha}(v_l)}^2} -\\
&& \sqrt{{\mathrm{deg}_{\G_\alpha}(u_1)}^2 + {\mathrm{deg}_{\G_\alpha}(u_r)}^2} - \sqrt{{\mathrm{deg}_{\G_\alpha}(u_1)}^2 + {\mathrm{deg}_{\G_\alpha}(w_1)}^2} -\\
&& \sqrt{{\mathrm{deg}_{\G_\alpha}(w_{k-1})}^2 + {\mathrm{deg}_{\G_\alpha}(v_1)}^2} \\
&\geq& \sqrt{3^2 + 2^2} + \sqrt{3^2 + 2^2} + \sqrt{3^2 + 2^2} + \sqrt{3^2 + 2^2} + \sqrt{3^2 + 2^2} + \sqrt{3^2 + 2^2} -\\
&& \sqrt{2^2 + 2^2} - \sqrt{3^2 + 2^2} - \sqrt{3^2 + 2^2} - \sqrt{3^2 + 2^2} - \sqrt{2^2 + 1^2} \\
&=& 3 \sqrt{13} - \sqrt{8} - \sqrt{5} > 0
\end{eqnarray*}
This completes the proof.
\end{proof}
\wbull \\

Now we introduce the $\beta$-operation by employing it to a $\G \in \mathbb{B}_n^k$.
Assume that $\G$ comprises vertices $u_1,\ldots,u_r$, $v_1,\ldots,v_l$
and $w_1,\ldots,w_\xi$, such that $l+\xi=k$ and $u_1,\ldots,u_r$ generate a cycle $C_r$, $v_i$
and $w_j$ ($1\leq i \leq l$ and $1\leq j \leq \xi$) induces paths, say, $P_l$ and $P_\xi$.
Let $\G_{\beta}=\G- u_2 v_1 + w_\xi v_1$. In this way, the resulting graph $\G_{\beta}$ 
obtained by applying the $\beta$-operation will be called the $\beta$-switched graph.
Figure \ref{Fig3} employs the $\beta$-operation on $\G \in \mathbb{B}_n^k$.
	
\begin{figure}[htbp!]
\centering
\includegraphics[scale=0.9]{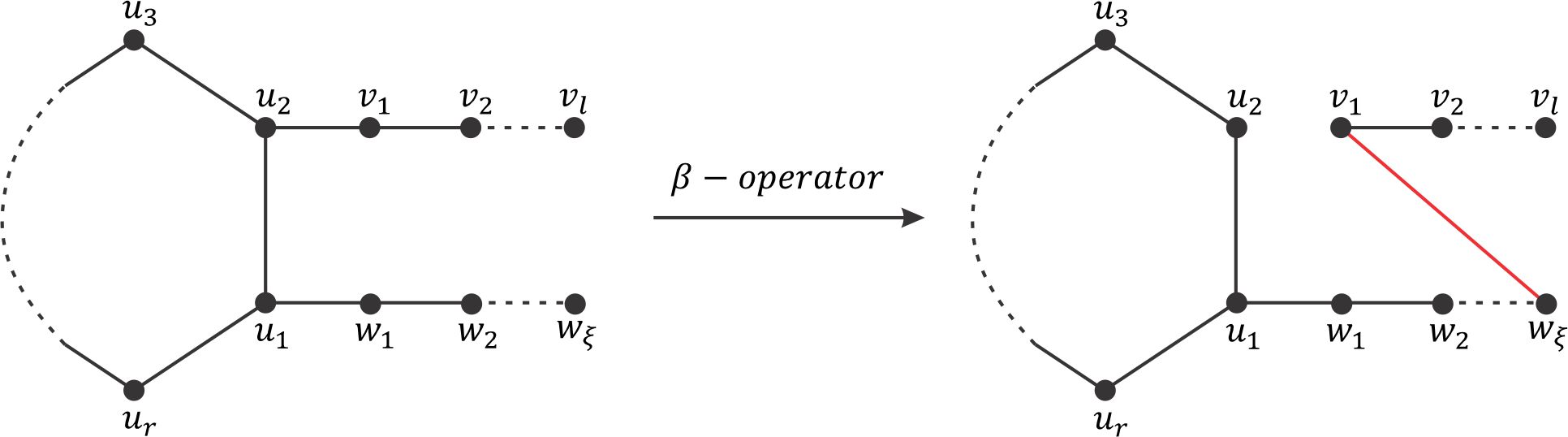}
\caption{Application of the $\beta$-operation on $\G \in \mathbb{B}_n^k$. It has been used in Lemma \ref{lemmabeta}.}\label{Fig3}
\end{figure}

\begin{lemma}\label{lemmabeta}
Assume $\G_{\beta}$ is obtained by applying the $\beta$-operation on $\G\in\mathbb{B}_n^k$ as explained in Figure \ref{Fig3}.
We obtain $$SO(\G_{\beta})<SO(\G).$$
\end{lemma}
\begin{proof}
The graph $\G$ has exactly one cycle $C_r = u_1 u_2 \dots u_r$ and two paths
$P_l = (v_1, v_2, \dots , v_l)$ and $P_\xi = (w_1, w_2, \dots , w_\xi)$, such
that $u_1 \sim w_1$ and $u_2 \sim v_1$. The graph $\G_\beta$ is obtained from
the graph $\G$ by deleting an edge $u_2 v_1$ and adding an edge $w_\xi v_1$. 
Between $SO(\G)$ and $SO(\G_{\beta})$, this shifting of edges generates 
the following relation:
\begin{eqnarray*}
SO(\G)-SO(\G_\beta) &=&  \sqrt{{\mathrm{deg}_{\G}(u_1)}^2 + {\mathrm{deg}_{\G}(u_2)}^2} + \sqrt{{\mathrm{deg}_{\G}(u_2)}^2 + {\mathrm{deg}_{\G}(u_3)}^2} + \sqrt{{\mathrm{deg}_{\G}(u_2)}^2 + {\mathrm{deg}_{\G}(v_1)}^2} +\\
&& \sqrt{{\mathrm{deg}_{\G}(w_{\xi-1})}^2 + {\mathrm{deg}_{\G}(w_\xi)}^2} - \sqrt{{\mathrm{deg}_{\G_\beta}(u_1)}^2 + {\mathrm{deg}_{\G_\beta}(u_2)}^2} -\\
&&\sqrt{{\mathrm{deg}_{\G_\beta}(u_2)}^2 + {\mathrm{deg}_{\G_\beta}(u_3)}^2} - \sqrt{{\mathrm{deg}_{\G_\beta}(w_\xi)}^2 + {\mathrm{deg}_{\G_\beta}(v_1)}^2} -\\
&& \sqrt{{\mathrm{deg}_{\G_\beta}(w_{\xi-1})}^2 + {\mathrm{deg}_{\G_\beta}(w_\xi)}^2} \\		
&\geq& \sqrt{3^2 + 3^2} + \sqrt{3^2 + 2^2} + \sqrt{3^2 + 2^2} + \sqrt{2^2 + 1^2} -\\
&& \sqrt{3^2 + 2^2} - \sqrt{2^2 + 2^2} - \sqrt{2^2 + 2^2} - \sqrt{2^2 + 2^2} \\
&=& \sqrt{13} + \sqrt{5} - 3\sqrt{2} > 0
\end{eqnarray*}
This completes the proof.		
\end{proof}
\wbull \\

Now we elaborate the $\upgamma$-operation by employing it to a $\G \in \mathbb{B}_n^k$.
Roughly speaking, the $\upgamma$-operator takes a graph having nested cycles, and transforms it into a unicyclic graph.
Assume that $\G$ comprises vertices $u_1,\ldots,u_l$, $v_{r+1}, v_{r+2}, \ldots,v_m$.
Any of these vertices could contain a subgraph attached to it.
Let $\G_{\upgamma}=\G- \{u_{r-1} u_r + u_1 u_l \}+ u_{r-1} u_l$. 
In this way, the resulting graph $\G_{\upgamma}$
obtained by applying the $\upgamma$-operation will be called the $\upgamma$-switched graph.
Figure \ref{Fig4} employs the $\upgamma$-operation on $\G \in \mathbb{B}_n^k$.
	
\begin{figure}[htbp!]
\centering
\includegraphics[scale=0.60]{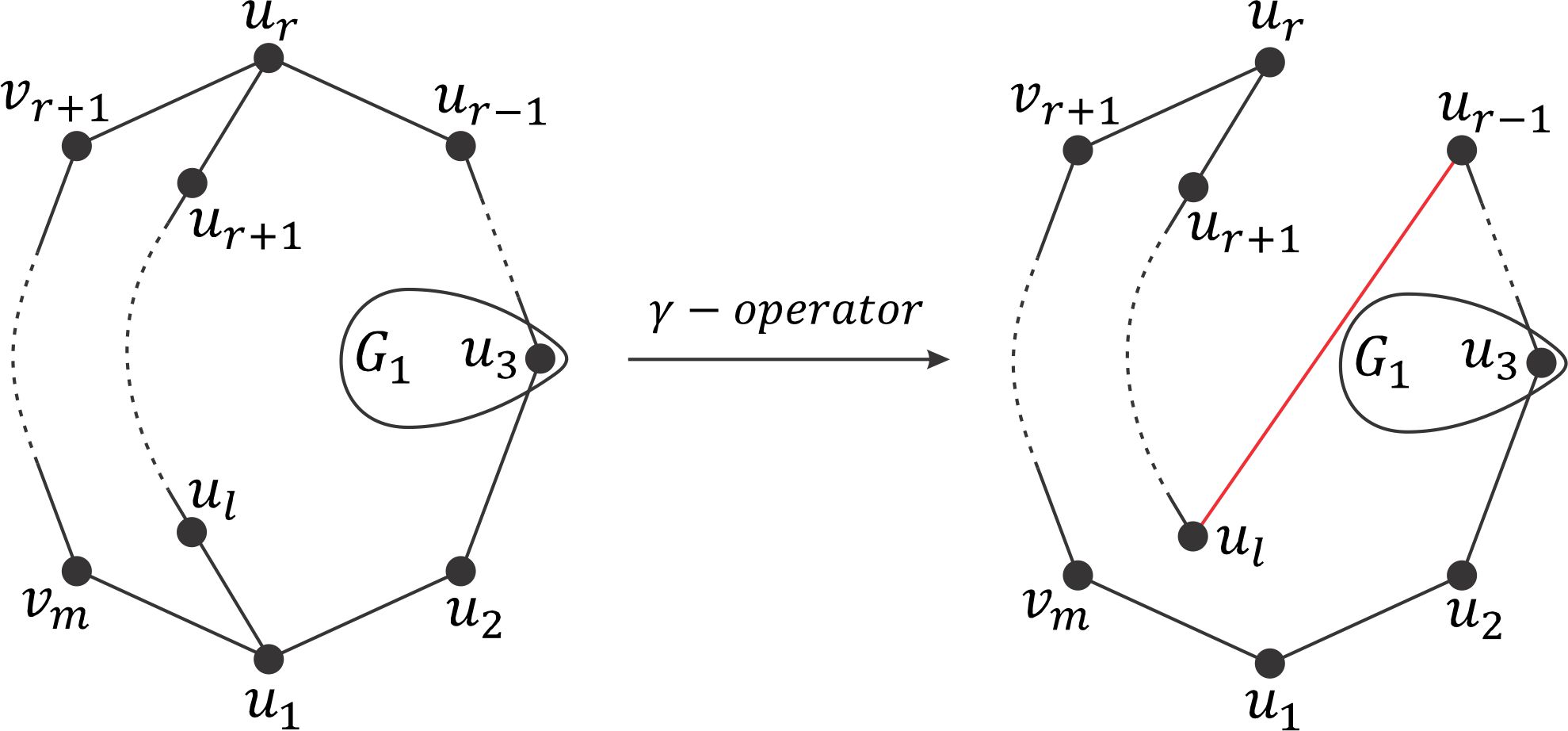}
\caption{Application of the $\upgamma$-operation on $\G \in \mathbb{B}_n^k$. It has been used in Lemma \ref{lemmagamma}.}\label{Fig4}
\end{figure}
	
\begin{lemma}\label{lemmagamma}
Assume $\G_\upgamma$ is obtained by applying the $\upgamma$-operation on $\G\in\mathbb{B}_n^k$ as explained in Figure \ref{Fig4}.
We obtain $$SO(\G_{\upgamma}) < SO(\G).$$
\end{lemma}
\begin{proof}
Without loss of generality, we may assume that $\G$ comprises two nested cycles, i.e. a bicyclic graph.
Figure \ref{Fig4} depicts such a case. The $\upgamma-$operator converts the bicyclic graph $\G$ into
the unicyclic $\G_\upgamma$. Between $SO(\G)$ and $SO(\G_{\upgamma})$, this shifting of edges generates
the following relation:
\begin{eqnarray*}
SO(\G)-SO(\G_\upgamma) &=&  \sqrt{{\mathrm{deg}_{\G}(u_1)}^2 + {\mathrm{deg}_{\G}(u_2)}^2} +
\sqrt{{\mathrm{deg}_{\G}(u_1)}^2 + {\mathrm{deg}_{\G}(u_l)}^2} +\\
&& \sqrt{{\mathrm{deg}_{\G}(u_1)}^2 + {\mathrm{deg}_{\G}(v_m)}^2} +
\sqrt{{\mathrm{deg}_{\G}(u_r)}^2 + {\mathrm{deg}_{\G}(u_{r-1})}^2} +\\
&& \sqrt{{\mathrm{deg}_{\G}(u_r)}^2 + {\mathrm{deg}_{\G}(u_{r+1})}^2} +
\sqrt{{\mathrm{deg}_{\G}(u_r)}^2 + {\mathrm{deg}_{\G}(v_{r+1})}^2} -\\
&& \sqrt{{\mathrm{deg}_{\G_\upgamma}(u_1)}^2 + {\mathrm{deg}_{\G_\upgamma}(u_2)}^2} -
\sqrt{{\mathrm{deg}_{\G_\upgamma}(u_1)}^2 + {\mathrm{deg}_{\G_\upgamma}(v_m)}^2} -\\
&& \sqrt{{\mathrm{deg}_{\G_\upgamma}(u_r)}^2 + {\mathrm{deg}_{\G_\upgamma}(u_{r+1})}^2} -
\sqrt{{\mathrm{deg}_{\G_\upgamma}(u_r)}^2 + {\mathrm{deg}_{\G_\upgamma}(v_{r+1})}^2} -\\
&& \sqrt{{\mathrm{deg}_{\G_\upgamma}(u_l)}^2 + {\mathrm{deg}_{\G_\upgamma}(u_{r-1})}^2} \\
&\geq& 6 \sqrt{3^2 + 2^2} - 5 \sqrt{2^2 + 2^2} > 0
\end{eqnarray*}
This completes the proof.
\end{proof}
\wbull \\

Next lemma gives some information about the structure of a graph $\G \in \mathbb{B}_n^k$ achieving a minimum Sombor index.
\begin{lemma}\label{uniquecycle}
Let $\G \in \mathbb{G}_n^{k}$ such that $k<n-1$ and $SO(\G) \leq SO(\G')$ for any $\G' \in \mathbb{B}_n^k$. Then, $\G$ has a unique cycle of length $n-k$.
\end{lemma}
\begin{proof}
Since $\G$ is not a tree, we obtain that it has at least one cycle. Next, we show that $\G$ is unicyclic.
Suppose, on contrary, that $\G$ has more than one cycle. First we show that all these cycles are edge-disjoint.
\begin{claim}\label{claim1}
All cycles in $\G$ are edge-disjoint.
\end{claim}
\begin{proof}
Assume $\G$ has cycles $C_1, C_2, \dots, C_s$ for $s \geq 2$ which share some edges. Then a successive
application of Lemma \ref{lemmagamma} implies that there exists $\G'$ implying $SO(\G)>SO(\G')$.
This contradicts the choice of $\G$ which, in turn, shows the claim.
\end{proof}
\wbull \\

By Claim \ref{claim1}, we conclude that all cycles in $\G$ are edge-disjoint.
Next, by successive application of Lemma \ref{lemmaalpha}, there exists
a graph $\G''$ implying $SO(\G)>SO(\G'')$. Thus, it again contradicts the choice of $\G$,
which shows that $\G$ is unicyclic, say, with cycle $C$.\\

Since $G$ is connected and has $k$ bridges, this implies by Theorem \ref{bridgecyclesthm} that
there are exactly $k$ edges which do not lies on $C$. This leads us to conclude that $C$ has length $n-k$.
\end{proof}
\wbull \\

Finally, we elaborate the $\delta$-operation by employing it to a $\G \in \mathbb{B}_n^k$.
Assume that $\G$ comprises
vertices $u_1, u_2, \dots, u_r$, $x_1, x_2, \dots, x_s$, $v_1, v_2, \dots, v_l$, and
$ w_1, w_2, \dots, w_\xi$ such that $u_1, u_2, \dots, u_r$ form a cycle of length $r=n-k$ and $s \geq 0$.
Let $\G_\delta = \G - x_s w_1 + v_l w_1$.
In this way, the resulting graph $\G_\delta$
obtained by applying the $\delta$-operation will be called the $\delta$-switched graph.
Figure \ref{Fig5} employs the $\delta$-operation on $\G \in \mathbb{B}_n^k$.
	
\begin{figure}[htbp!]
\centering
\includegraphics[scale=0.70]{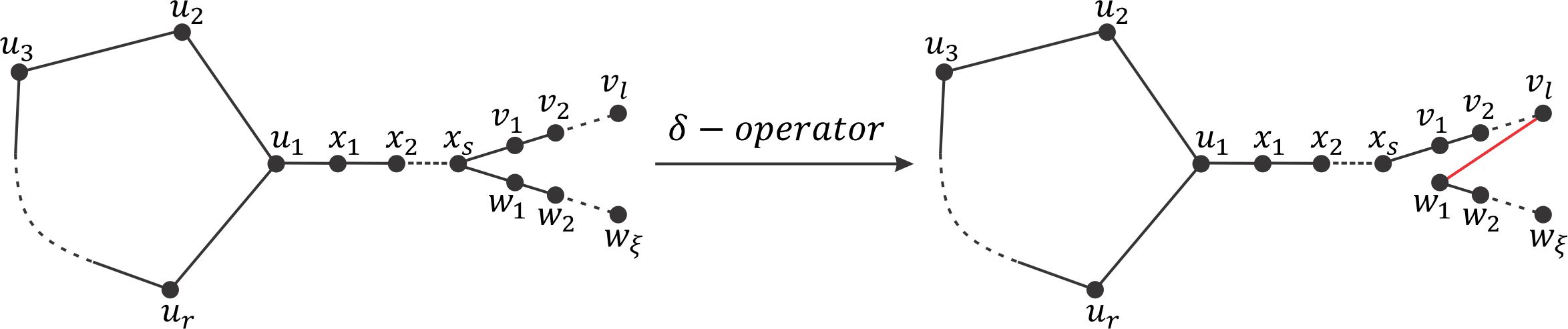}
\caption{Application of the $\delta$-operation on $\G \in \mathbb{B}_n^k$. It has been used in Lemma \ref{lemmadelta}.}\label{Fig5}
\end{figure}
	
\begin{lemma}\label{lemmadelta}
Assume $\G_\delta$ is obtained by applying the $\delta$-operation on $\G\in\mathbb{B}_n^k$ as explained in Figure \ref{Fig5}.
We have $$SO(\G_{\delta}) < SO(\G).$$
\end{lemma}
\begin{proof}
The graph $\G$ assumes a cycle $C_r = u_1 u_2 \dots u_r$, a path $u_1, x_1, x_2, \dots, x_s$
where $s=0,1,2,\dots,k$ and two pendant paths, say $v_1, v_2, \dots, v_l$ and $w_1, w_2, \dots, w_\xi$,
attached with $x_s$ ($u_1$ if $s=0$). Note that $s+l+\xi=k$. Between $SO(\G)$ and $SO(\G_{\delta})$, this shifting of edges generates
the following relation:
\begin{eqnarray*}
	SO(\G)-SO(\G_\delta) &=&  \sqrt{{\mathrm{deg}_{\G}(x_{s-1})}^2 + {\mathrm{deg}_{\G}(x_s)}^2}
	+ \sqrt{{\mathrm{deg}_{\G}(v_1)}^2 + {\mathrm{deg}_{\G}(x_s)}^2} +\\
	& &\sqrt{{\mathrm{deg}_{\G}(w_1)}^2 + {\mathrm{deg}_{\G}(x_s)}^2} +
	\sqrt{{\mathrm{deg}_{\G}(v_{l-1})}^2 + {\mathrm{deg}_{\G}(v_l)}^2} -\\
	& & \sqrt{{\mathrm{deg}_{\G_\delta}(x_{s-1})}^2 + {\mathrm{deg}_{\G_\delta}(x_s)}^2} -
	\sqrt{{\mathrm{deg}_{\G_\delta}(v_1)}^2 + {\mathrm{deg}_{\G_\delta}(x_s)}^2} -\\
	& &  \sqrt{{\mathrm{deg}_{\G_\delta}(v_{l-1})}^2 + {\mathrm{deg}_{\G_\delta}(v_l)}^2} -
	\sqrt{{\mathrm{deg}_{\G_\delta}(v_l)}^2 + {\mathrm{deg}_{\G_\delta}(w_1)}^2}	\\
	&\geq& \sqrt{2^2 + 3^2} + \sqrt{2^2 + 3^2} + \sqrt{2^2 + 3^2} + \sqrt{2^2 + 1^2} -\\
	&& \sqrt{2^2 + 2^2} - \sqrt{2^2 + 2^2} - \sqrt{2^2 + 2^2} - \sqrt{2^2 + 2^2} \\
	&=& 3 \sqrt{13} - 8 \sqrt{2} + \sqrt{5} >0
\end{eqnarray*}
This completes the proof.
\end{proof}
\wbull

Next, we present the main result of this paper.
\section{Sombor index with a given number of bridges of graphs}\label{mainresults}
This section employs auxiliary results from Section \ref{auxualiary} to evaluate a
lower bound on the Sombor index in $\mathbb{B}_n^k$. Graphs attaining the lower bound
have also been classified. For $n \geq3$ and $k = 0, 1, 2, \dots, n-3, n-1$,
let $C_{n-k}$ be a cycle having length $n-k$ and $u_1$ be an arbitrary vertex of $C_{n-k}$. 
We obtain the family $\PP_n^k$ by attaching a path $P_{k+1}$ on $k+1$ vertices to $u_1\in V(C_{n-k})$. Thus, it implies that $\PP_n^k$
a family belonging to $\mathbb{B}_n^k$ having order $n$ and $k$ number of bridges. See Figure \ref{Fig6} for a depiction of
the family $\mathbb{B}_n^k$. It is worth noticing that for an $n$-vertex in $\mathbb{B}_n^k$, we have either $k=n-1$ or
$0\leq k\leq n-3$.
	
\begin{figure}[htbp!]
\centering
\includegraphics[scale=0.70]{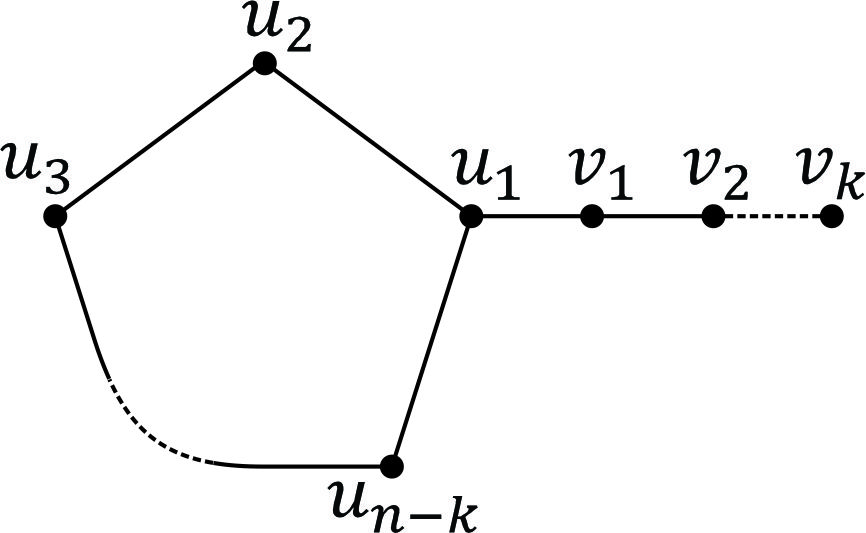}
\caption{The graph $\PP_n^k$.}
\label{Fig6}
\end{figure}

Finally, we have sufficient tools to present the main result.
\begin{theorem}\label{CycleThm}
Assume that $\G$ is a bridgeless graph of order $n$, i.e. $\G \in \mathbb{B}_n^k$ such that $k=0$. Then
$$SO(\G) \geq \sqrt{8}~n,$$
where equality is satisfied $\iff$ $\G \cong C_n$.
\end{theorem}
\begin{proof}
Since $\G$ bridgeless, by Theorem \ref{bridgecyclesthm}, each edge of $\G$ lies on some cycles of $\G$.
Thus $\mathrm{deg}_{\G}(v) \geq 2$ for any $v \in V(\G)$. Since $\G$ is connected and contains at least one cycle,
by Proposition \ref{connectedgraphwithcycles}, we obtain that $m\geq n$. Therefore, we have:
\begin{eqnarray*}
SO(\G) &=& \sum _{uv\in E\left( \G\right) }\sqrt{{\mathrm{deg}_{\G}(u)}^{2}+{\mathrm{deg}_{\G}(v)}^{2}} \\
&\geq& \sum _{uv\in E\left( \G\right) }\sqrt{{2}^{2}+{2}^{2}} = \sqrt{8}~m \geq \sqrt{8}~n.
\end{eqnarray*}
By Lemma \ref{uniquecycle}, the graph $\G \in \mathbb{G}_n^0$ achieving the minimum has a unique cycle of length
$n$, thus we obtain that $\G\cong C_n$, since we have $k=0$.\\

Routine calculations show that $SO(C_n)=\sqrt{8}~n$ which, in turn, implies that the equality is 
satisfied $\iff$ $\G\cong C_n\cong\PP_n^0$.
\end{proof}
\wbull \\

Next, we show a sharp lower bound for graphs in $\mathbb{B}_n^k$ where $k\geq1$.
In Theorem \ref{OnebridgeThm}, we present the result for graphs in $\mathbb{B}_n^k$
with $k=1$.
\begin{theorem}\label{OnebridgeThm}
Assume that $\G$ is a graph in $\mathbb{G}_n^1$. So, we have
$$SO(\G) \geq 2(n-3) \sqrt{2} + \sqrt{10} + 2\sqrt{13},$$
such that equality is satisfied $\iff$ $\G \cong \PP_n^1$.
\end{theorem}
\begin{proof}
Let $\G' \in \mathbb{G}_n^1$ such that $SO(\G') \leq SO(\G)$ for any $\G \in \GG_n^1$. Then, by
Lemma \ref{uniquecycle}, $\G'$ has a unique cycle of length $n-1$, which implies that $\G' \cong \PP_n^1$.
Since $SO(\PP_n^1)=2(n-3) \sqrt{2} + \sqrt{10} + 2\sqrt{13}$, we obtain that
$SO(\G) \geq 2(n-3) \sqrt{2} + \sqrt{10} + 2\sqrt{13}$ for any $\G \in \GG_n^1$.
This proves the result.
\end{proof}
\wbull \\

A similar sharp lower bound for
$2\leq k\leq n-3$ has been shown herewith.
\begin{theorem}\label{morebridgesThm}
Assume that $\G$ is a graph in $\mathbb{B}_n^k$ for $2\leq k\leq n-3$. Thus, we have
$$SO(\G) \geq 2(n-4) \sqrt{2} + \sqrt{5} + 3\sqrt{13},$$
such that equality is satisfied $\iff$ $\G \cong \PP_n^k$.
\end{theorem}
\begin{proof}
Let $\G' \in \mathbb{B}_n^k$ such that $SO(\G') \leq SO(\G)$ for any $\G \in \mathbb{B}_n^k$, then by
Lemma \ref{uniquecycle}, $\G'$ has a unique cycle of length $n-k$. By a successive application
of Lemma \ref{lemmabeta} and \ref{lemmadelta}, we obtain that $\G' \cong \PP_n^k$. Routine
calculations show that $SO(\PP_n^k)= 2(n-4) \sqrt{2} + \sqrt{5} + 3\sqrt{13}$. Thus, we have
$SO(\G) \geq 2(n-4) \sqrt{2} + \sqrt{5} + 3\sqrt{13}$ for any $\G \in \mathbb{B}_n^k$, where $2\leq k\leq n-3$.
This completes the proof.
\end{proof}
\wbull \\

And finally, we show a sharp lower bound for $\G \in \mathbb{B}_n^k$ with $k=n-1$.
Note that by Corollary \ref{Tree_bridge}, $\G$ is a tree.
\begin{theorem}\label{TreeThm}
For an $n$-vertex tree, we obtain
$$SO(T) \geq 2(n-3)\sqrt{2}+2\sqrt{5},$$
such that the equality is satisfied $\iff$ $T \cong P_n \cong \PP_n^{n-1}$.
\end{theorem}
\begin{proof}
First, we show the following claim for a tree $T\in\mathbb{B}^n_k$ achieving the
minimum Sombor index.
\begin{claim}\label{claim0}
$1\leq\mathrm{deg}_{T}(u)\leq2$ for any $u\in V(T)$.
\end{claim}
\noindent{\em Proof of Claim \ref{claim0}.}
We show the claim by contradiction. We assume the existence of a $y\in V(T)$, on contrary, satisfying
$\mathrm{deg}_{T}(y)\geq3$. A successive application of Lemma \ref{lem00} by assuming $y=u_k\in V(T_{\tau})$
(see Figure \ref{Fig1}) generates a tree $T'$ satisfying $SO(T)>SO(T')$. However, $T$ was assumed to possess
minimum Sombor index. This arises a contradiction. This proves the claim.
\wbull \\
		
Claim \ref{claim0} implies that $T \cong P_n$, as $T$ is connected.
One could calculate that $SO(P_n) = 2(n-3)\sqrt{2}+2\sqrt{5}$, which
completes the proof.
\end{proof}
\wbull \\

Now, in view of Theorems \ref{CycleThm}, \ref{OnebridgeThm}, \ref{morebridgesThm}, and  \ref{TreeThm},
we conclude the main result as follows.
\begin{theorem}
Let $\G \in \mathbb{B}_n^k$. Then
\begin{equation*}
SO(\G) \geq
\left\{\begin{split}
                 &\sqrt{8}~n; && \text{for} \quad k=0, \\
				 &2(n-3) \sqrt{2} + \sqrt{10} + 2\sqrt{13}; && \textrm{for} \quad k=1, \\
				 &2(n-4) \sqrt{2} + \sqrt{5} + 3\sqrt{13}; && \textrm{for} \quad 2\leq k\leq n-3, \\
				 &2(n-3)\sqrt{2}+2\sqrt{5}; && \textrm{for} \quad k=n-1,
\end{split}\right.
\end{equation*}
such that the equality is satisfied $\iff$ $\G \cong \PP_n^k$.
\end{theorem}

\section{Sombor index of graphs with given vertex-connectivity}

In this section, we give a sharp upper bound on $SO(\G)$ for graphs in $\mathbb{V}_n^k$.
For this we need the following crucial result:
\begin{lemma}\label{pq1}
Let $k_1,\,k,\,k_2$ and $n$ be positive integers such that $n=k_1+k+k_2$ with $k_1\geq k_2$. Also let
\begin{align*}
f(k_1,\,k_2)&=\sqrt{2}\,{k_1\choose 2}\,(k+k_1-1)+\sqrt{2}\,{k\choose 2}\,(n-1)+\sqrt{2}\,{k_2\choose 2}\,(k+k_2-1)\\[2mm]
&~~~~~~~~~~+kk_1\,\sqrt{(n-1)^2+(k+k_1-1)^2}+kk_2\,\sqrt{(n-1)^2+(k+k_2-1)^2}.
\end{align*}
Then $f(k_1+1,\,k_2-1)-f(k_1,\,k_2)>0$.
\end{lemma}

\begin{proof} We have
\begin{align}
&f(k_1+1,\,k_2-1)-f(k_1,\,k_2)\nonumber\\[2mm]
=&\sqrt{2}\,\left[{k_1+1\choose 2}\,(k+k_1)-{k_1\choose 2}\,(k+k_1-1)\right]+\sqrt{2}\,\left[{k_2-1\choose 2}\,(k+k_2-2)-{k_2\choose 2}\,(k+k_2-1)\right]\nonumber\\[3mm]
&+k\,(k_1+1)\,\sqrt{(n-1)^2+(k+k_1)^2}-kk_1\,\sqrt{(n-1)^2+(k+k_1-1)^2}\nonumber\\[3mm]
&+k\,(k_2-1)\,\sqrt{(n-1)^2+(k+k_2-2)^2}-kk_2\,\sqrt{(n-1)^2+(k+k_2-1)^2}.\label{1kin1}
\end{align}
Now,
\begin{align}
{k_1+1\choose 2}\,(k+k_1)-{k_1\choose 2}\,(k+k_1-1)&=\frac{k_1}{2}\,\Big[(k_1+1)\,(k+k_1)-(k_1-1)\,(k+k_1-1)\Big]\nonumber\\
   &=\frac{k_1}{2}\,(2k+3k_1-1)\label{1kin2}
\end{align}
and
\begin{align}
{k_2-1\choose 2}\,(k+k_2-2)-{k_2\choose 2}\,(k+k_2-1)&=-\frac{k_2-1}{2}\,\Big[k_2\,(k+k_2-1)-(k_2-2)\,(k+k_2-2)\Big]\nonumber\\
   &=\frac{k_2-1}{2}\,(2k+3k_2-4),\label{1kin3}
\end{align}

\vspace*{3mm}

\noindent
${\bf Claim\,1.}$
$$\sqrt{(n-1)^2+(k+k_1)^2}\,\sqrt{(n-1)^2+(k+k_2-2)^2}>\sqrt{(n-1)^2+(k+k_2-1)^2}\,\sqrt{(n-1)^2+(k+k_1-1)^2}.$$

\vspace*{3mm}

\noindent
${\bf Proof\,of\,Claim\,1.}$ We have to prove that
$$\Big[(n-1)^2+(k+k_1)^2\Big]\,\Big[(n-1)^2+(k+k_2-2)^2\Big]>\Big[(n-1)^2+(k+k_2-1)^2\Big]\,\Big[(n-1)^2+(k+k_1-1)^2\Big],$$
that is,
\begin{align*}
&(n-1)^2\,\Big[(k+k_1)^2+(k+k_2-2)^2\Big]+(k+k_1)^2\,(k+k_2-2)^2>(n-1)^2\,\Big[(k+k_1-1)^2+(k+k_2-1)^2\Big]\\
&~~~~~~~~~~~~~~~~~~~~~~~~~~~~~~~~~~~~~~~~~~~~~~~~~~~~~~~~~~~~~~~~~~~+(k+k_1-1)^2\,(k+k_2-1)^2,
\end{align*}
that is,
$$2(n-1)^2\,(k_1-k_2+1)>\Big[(k+k_1)^2-2\,(k+k_1)+1\Big]\,(k+k_2-1)^2-(k+k_1)^2\,\Big[(k+k_2-1)^2-2\,(k+k_2-1)+1\Big],$$
that is,
$$2(n-1)^2\,(k_1-k_2+1)>(k+k_2-1)^2-(k+k_1)^2+2\,(k+k_1)\,(k+k_2-1)\,(k_1-k_2+1),$$
which is true always as $k_1\geq k_2$, that is, $k+k_2-1<k+k_1$ and $n-1>k+k_1>k+k_2-1$. This proves the {\bf Claim 1}.

\vspace*{3mm}

\noindent
${\bf Claim\,2.}$
\begin{align*}
&\sqrt{(n-1)^2+(k+k_1)^2}-\sqrt{(n-1)^2+(k+k_1-1)^2}\\
&~~~~~~~~~~~~~~~~~~~~~~~~~~~~~~~~>\sqrt{(n-1)^2+(k+k_2-1)^2}-\sqrt{(n-1)^2+(k+k_2-2)^2}.
\end{align*}

\vspace*{3mm}

\noindent
${\bf Proof\,of\,Claim\,2.}$ We have to prove that
\begin{align*}
&\sqrt{(n-1)^2+(k+k_1)^2}+\sqrt{(n-1)^2+(k+k_2-2)^2}\\
&~~~~~~~~~~~~~~~~~~~~~~~~~~~~~~~~>\sqrt{(n-1)^2+(k+k_2-1)^2}+\sqrt{(n-1)^2+(k+k_1-1)^2},
\end{align*}
that is,
\begin{align*}
&(n-1)^2+(k+k_1)^2+(n-1)^2+(k+k_2-2)^2+2\,\sqrt{(n-1)^2+(k+k_1)^2}\,\sqrt{(n-1)^2+(k+k_2-2)^2}\\
&~~~~>(n-1)^2+(k+k_2-1)^2+(n-1)^2+(k+k_1-1)^2+2\,\sqrt{(n-1)^2+(k+k_2-1)^2}\\
&~~~~~~~~~~~~~~~~~~~~~~~~~~~~~~~\times\sqrt{(n-1)^2+(k+k_1-1)^2},
\end{align*}
that is,
\begin{align*}
&k_1-k_2+1+\sqrt{(n-1)^2+(k+k_1)^2}\,\sqrt{(n-1)^2+(k+k_2-2)^2}\\
&~~~~~~~~~~~~~~~~~~~~~~~~~~~-\sqrt{(n-1)^2+(k+k_2-1)^2}\,\sqrt{(n-1)^2+(k+k_1-1)^2}>0,
\end{align*}
which is true as $k_1\geq k_2$ and by {\bf Claim 1}. This proves the {\bf Claim 2}.

\vspace*{3mm}

Using (\ref{1kin2}), (\ref{1kin3}) with $k_1\geq k_2$ and {\bf Claim 2} in (\ref{1kin1}), we obtain
\begin{align*}
&~~f(k_1+1,\,k_2-1)-f(k_1,\,k_2)\\[2mm]
&=\frac{k_1}{2}\,(2k+3k_1-1)-\frac{(k_2-1)}{2}\,(2k+3k_2-4)+kk_1\,\Big[\sqrt{(n-1)^2+(k+k_1)^2}-\sqrt{(n-1)^2+(k+k_1-1)^2}\Big]\\[3mm]
&~~~~~~~~+kk_2\,\Big[\sqrt{(n-1)^2+(k+k_2-2)^2}-\sqrt{(n-1)^2+(k+k_2-1)^2}\Big]\\[3mm]
&~~~~~~~~+k\,\Big[\sqrt{(n-1)^2+(k+k_1)^2}-\sqrt{(n-1)^2+(k+k_2-2)^2}\Big]\\[3mm]
&\geq\frac{k_2}{2}\,(2k+3k_2-1)-\frac{(k_2-1)}{2}\,(2k+3k_2-4)+kk_2\,\Big[\sqrt{(n-1)^2+(k+k_1)^2}-\sqrt{(n-1)^2+(k+k_1-1)^2}\\[3mm]
&~~~~~~~~+\sqrt{(n-1)^2+(k+k_2-2)^2}-\sqrt{(n-1)^2+(k+k_2-1)^2}\Big]\\[3mm]
&~~~~~~~~+k\,\Big[\sqrt{(n-1)^2+(k+k_2)^2}-\sqrt{(n-1)^2+(k+k_2-2)^2}\Big]>0
\end{align*}
as $\frac{k_1}{2}\,(2k+3k_1-1)>\frac{(k_2-1)}{2}\,(2k+3k_2-4)$ and $\sqrt{(n-1)^2+(k+k_2)^2}-\sqrt{(n-1)^2+(k+k_2-2)^2}>0$.
This completes the proof of the result.
\end{proof}
\wbull \\

Denote by $\mathbb{V}_n^k$ the set of $n$-vertex connected graphs with vertex-connectivity $k$, where $1\leq k\leq n-1$.
\begin{theorem}\label{thmconnectivity}
Let $\G\in \mathbb{V}_n^k$ with $1\leq k\leq n-1$. Then
$$SO(\G)\leq \sqrt{2}\,{n-k-1\choose 2}\,(n-2)+\sqrt{2}\,{k\choose 2}\,(n-1)+k\,(n-k-1)\,\sqrt{(n-1)^2+(n-2)^2}+k\,\sqrt{(n-1)^2+k^2}$$
with equality if and only if $\G\cong (K_{n-k-1}\cup K_1)\vee K_k$.
\end{theorem}
\begin{proof}
When $k=n-1$, there is only one graph $K_n$, which can be viewed as a special case of $(K_{n-k-1}\cup K_1)\vee K_k$ with $k=n-1$. We now assume that $1\leq k\leq n-2$. Let $\Om$ be a graph in $\mathbb{V}_n^k$ such that $SO(\G)\leq SO(\Om)$. Since $k$ is the vertex connectivity of graph $\Om$, let $S=\{v_1,\,v_2,\ldots,\,v_k\}\subseteq V(\Om)$ be the vertex cut and $|S|=k$. By Lemma \ref{lem0}, we obtain that $\Om[S]$ must be a complete graph $K_k$.
We now prove that $\Om-S$ has exactly two components. To the contrary, assume that there exist at least three components $\Om_1$, $\Om_2$ and $\Om_3$ with $v_i\in V(\Om_i)$ for $i=1,\,2$. Then we find that $\Om+v_1v_2$ is a connected graph of order $n$ with vertex connectivity $k$, with a higher Sombor index than that of
$\Om$ by Lemma \ref{lem0}, contradicting to the choice of $\Om$. Now we assume that $\Om-S=\Om_1\cup \Om_2$, where $\Om_1$ and $\Om_2$ are the two
components of $\Om-S$. Again by Lemma \ref{lem0}, we conclude that $\Om_1$ and $\Om_2$ are all cliques and each vertex in $S$ is incident
to all vertices in $\Om_1\cup \Om_2$. Hence $\Om\cong (K_{k_1}\cup K_{k_2})\vee K_k$, where $k_1+k_2=n-k$. Without loss of generality, we can assume that $k_1\geq k_2$. Then
\begin{align*}
SO(\Om)&=\sum _{uv\in E\left(\Om\right) }\sqrt{{\mathrm{deg}_{\Om}(u)}^{2}+{\mathrm{deg}_{\Om}(v)}^{2}} \\[2mm]
&=\sqrt{2}\,{k_1\choose 2}\,(k+k_1-1)+\sqrt{2}\,{k\choose 2}\,(n-1)+\sqrt{2}\,{k_2\choose 2}\,(k+k_2-1)\\[2mm]
&~~~~~~~~~~+kk_1\,\sqrt{(n-1)^2+(k+k_1-1)^2}+kk_2\,\sqrt{(n-1)^2+(k+k_2-1)^2}=f(k_1,\,k_2),~\mbox{(say)}.
\end{align*}
If $k_2=1$, then $\Om\cong (K_{n-k-1}\cup K_{1})\vee K_k$ and hence the equality holds. Otherwise, $k_2\geq 2$. By Lemma \ref{pq1}, we obtain
$$f(k_1,\,k_2)<f(k_1+1,\,k_2-1)<\cdots<f(k_1+k_2-2,2)<f(k_1+k_2-1,1).$$
From the above results, we obtain
\begin{align*}
SO(\G)\leq SO(\Om)\leq f(k_1,\,k_2)&<f(k_1+k_2-1,1)\\
&=\sqrt{2}\,{n-k-1\choose 2}\,(n-2)+\sqrt{2}\,{k\choose 2}\,(n-1)\\
&~~~~~+k\,(n-k-1)\,\sqrt{(n-1)^2+(n-2)^2}+k\,\sqrt{(n-1)^2+k^2}.
\end{align*}
This completes the proof of the theorem.
\end{proof}
\wbull \\

\section{Conclusions}
This paper studies the Sombor index of $n$-vertex graphs having $k$ bridges i.e. $\mathbb{B}_n^k$.
Recently in 2021, Horoldagva \& Xu \cite{HX2021} found a sharp upper bound on the Sombor index of graphs in $\mathbb{B}_n^k$
As a continuance of this study, in this paper,
we find a sharp lower bound on the Sombor index of graphs in $\mathbb{B}_n^k$.
Moreover, graphs achieving the lower bound have been characterized.
Moreover, we find a sharp upper bound on the Sombor index of graphs in $\mathbb{V}_n^k$. Corresponding
extremal graphs achieving the bound have been characterized.

Let $\mathbb{E}_n^k$ be the set of all $n$-vertex graphs with edge connectivity at most $k$.
Here we would like to remark the following:
\begin{remark}
The problem of finding maximum Sombor index
of graphs in $\mathbb{E}_n^k$ is similar to Theorem \ref{thmconnectivity} and the same
unique graph i.e. $(K_{n-k-1}\cup K_1)\vee K_k$ is expected to achieve the maximum.
\end{remark}

Having said that, we propose the following open problems on the Sombor index of graphs.
\begin{problem}
Find minimal graphs with respect to the Sombor index among graphs in $\mathbb{V}_n^k$.
\end{problem}

\begin{problem}
Find minimal graphs with respect to the Sombor index among graphs in $\mathbb{E}_n^k$.
\end{problem}

 \vspace*{4mm}

\noindent
{\it Acknowledgement.} K. C. Das is supported by National Research Foundation funded by the Korean government (Grant No. 2021R1F1A1050646).

\end{document}